\documentclass[12pt,plain]{article}
\usepackage{fancyhdr}
\usepackage{verbatim}
\usepackage{indentfirst}
\usepackage{graphicx}
\usepackage{epstopdf}
\usepackage{color}
\usepackage{newlfont}
\usepackage{amssymb}
\usepackage[intlimits]{amsmath}
\usepackage{latexsym}
\usepackage{amsthm}
\usepackage{amscd}
\usepackage{fullpage}
\usepackage{setspace}
\usepackage{multirow}
\usepackage{hyperref}
\usepackage[dvips, hmargin=2.1cm,vmargin=3cm]{geometry}

\newcommand{\de}{\mathrm{d}}

\newcommand{\R}{\mathbb{R}}

\newcommand{\ha}{\frac{1}{2}}

\newcommand{\be}{\begin{equation}}
\newcommand{\bey}{\begin{eqnarray}}
\newcommand{\ee}{\end{equation}}
\newcommand{\eey}{\end{eqnarray}}
\newcommand{\ba}{\begin{array}}
\newcommand{\ea}{\end{array}}

\newcommand{\sgn}{\mathrm{sgn}}

\theoremstyle{plain}                    
\newtheorem{theorem}{Theorem}[section]     
\newtheorem{lem}[theorem]{Lemma}            

\newtheorem{cor}[theorem]{Corollary}

\theoremstyle{definition}

\newcommand{\E}{\mathbb{E}}

\begin{document}
\title{Non-Standard Limit Theorems in Number Theory}
\author{F. Cellarosi$^*$, Ya.G. Sinai$^{*,\dag}$}

\maketitle

\let\oldthefootnote\thefootnote
\renewcommand{\thefootnote}{\fnsymbol{footnote}}
\footnotetext[1]{Mathematics Department, Princeton University. Princeton, NJ, U.S.A.}
\footnotetext[2]{Landau Institute of Theoretical Physics, Russian Academy of Sciences. Moscow, Russia.}
\let\thefootnote\oldthefootnote

\begin{flushright}
\emph{Dedicated to Yu. V. Prokhorov\\ on the occasion of his 80-th birthday}
\end{flushright}

There are many unusual limit theorems in Number Theory which are well-known to experts in the field but not so well-known to probabilists. The purpose of this paper is to discuss some examples of such theorems. They were chosen in order to be close to the field of interest of Yu.V. Prokhorov.

One of the main objects in Number Theory is the so-called M\"obius function. It is defined as follows
\be
\mu(n)=\begin{cases}1&\mbox{if $n$=1;}\\0&\mbox{if $n$ is not square-free;}\\ (-1)^k&\mbox{if $n$ is the product of $k$ distinct primes.}\end{cases}
\nonumber\ee
Throughout the paper, when we write $n=p_1p_2\cdots p_k$ we assume that $p_1<p_2<\ldots<p_k$ are the first $k$  prime numbers.
Many properties of the M\"obius function are connected with the Riemann zeta function. For example, while the Prime Number Theorem is equivalent to the fact that $$\sum_{n\leq N}\mu(n)=o(N),$$ the Riemann Hypothesis is equivalent to 
$$\sum_{n\leq N}\mu(n)=O_\varepsilon\!\left(N^{1/2+\varepsilon}\right)$$ for every $\varepsilon>0$.

Recently, a conjecture by Sarnak \cite{Sarnak-Mobius} has fostered a great interest towards the connections between the M\"{o}bius function and Ergodic Theory, and in particular the works of Furstenberg \cite{Furstenberg1967} and Green and Tao \cite{Green-Tao-Mobius}.
\section{A probabilistic model for square-free numbers}

Fix $m>1$ and introduce the set $\Omega_m$, whose elements have the form $n=\prod_{j=1}^mp_j^{\nu_j}$, where $\nu_j\in\{0,1\}$. Then $\mu(n)=\pm1$ iff $n\in\Omega_m$ for some $m$. Define on $\Omega_m$ the probability distribution $\Pi_m$ for which 
\be\pi_m(n)=\frac{1}{Z_m}\frac{1}{n}=\frac{1}{Z_m\prod_{j=1}^mp_j^{\nu_j}},\label{1}\ee
In (\ref{1})
 $Z_m$ is the normalizing factor and
\be Z_m=\sum_{\nu_1,\ldots,\nu_m}\frac{1}{\prod_{j=1}^m p_j^{\nu_j}}=\prod_{j=1}^m\left(1+\frac{1}{p_j}\right)=\exp\left\{\sum_{j=1}^m\ln\left(1+\frac{1}{p_j}\right)\right\}=\exp\left\{O(1)+\sum_{j=1}^m\frac{1}{p_j}\right\}\nonumber\ee
as $m\to\infty$.
Denote by $N(t)$ the number of primes which are less or equal than $t$. The Prime Number Theorem says that $N(t)\sim\frac{t}{\ln t}$ as $t\to\infty$ and a slightly stronger version asserts that 
\be N(t)-\frac{t}{\ln t}=O\!\left(\frac{t}{\ln^2t}\right)\label{strongPNT}.\ee
We can write, by summation by parts, 
\bey&&\sum_{j=1}^m\frac{1}{p_j}=\sum_{t=1}^{p_m}\frac{1}{t}\left(N(t)-N(t-1)\right)=\frac{N(p_m)}{p_m+1}+\sum_{t=1}^{p_m}\frac{N(t)}{t(t+1)}=\nonumber\\
&&=\frac{m}{p_m+1}+\sum_{t=1}^{p_m}N(t)\left(\frac{1}{t^2}+O\!\left(\frac{1}{t^3}\right)\right)=O(1)+\sum_{t=2}^{p_m}\frac{1}{t\ln t}=O(1)+\ln\ln p_m,\nonumber
\nonumber\eey

i.e. $Z_m\sim O(1)\ln p_m$. A more precise asymptotic follows from Mertens' product formula \cite{Mertens-1874}
\be\lim_{n\to\infty}\ln n\prod_{p\leq n}\left(1-\frac{1}{p}\right)=e^{-\gamma}\approx0.561459,\nonumber\ee
where $\gamma$ is Euler-Mascheroni constant. In fact
\be
\frac{1}{\ln n}\prod_{p\leq n}\left(1+\frac{1}{p}\right)=
\frac{\left(
\prod_{p\leq n}
\frac{1}{1-p^{-2}}\right)^{-1}
}
{\ln n
\prod_{p\leq n}
\left(1-\frac{1}{p}\right)
}\longrightarrow\frac{\zeta(2)^{-1}}{e^{-\gamma}}\hspace{.5cm}\mbox{as $n\to\infty$.}
\nonumber\ee
 Thus
\be Z_m\sim \frac{e^\gamma}{\zeta(2)}\ln p_m\label{asymptotic-Zm}
.\ee
By analogy with Statistical Physics, $Z_m$ is called \emph{partition function}.

It is easy to check that w.r.t. $\Pi_m$, the random variables $\nu_j$ are independent and
\be\Pi_m\{\nu_j=0\}=\frac{p_j}{1+p_j},\hspace{.5cm}\Pi_m\{\nu_j=1\}=\frac{1}{1+p_j},\hspace{.5cm}1\leq j\leq m.\nonumber\ee
Indeed,
\bey\Pi_m\{\nu_j=0\}&=&\frac{1}{Z_m}\sum_{\nu_1,\ldots,\nu_{j-1}}\sum_{\nu_{j+1},\ldots,\nu_m}\frac{1}{\prod_{l=1}^{j-1}p_l^{\nu_l}\prod_{l=j+1}^mp_l^{\nu_l}}=\nonumber\\
&=&\frac{\prod_{l=1}^{j-1}\left(1+\frac{1}{p_l}\right)\prod_{l=j+1}^{m}\left(1+\frac{1}{p_l}\right)}{\prod_{l=1}^{m}\left(1+\frac{1}{p_l}\right)}=\frac{p_j}{1+p_j}.\nonumber\eey
Since 
\be n=\prod_{j=1}^mp_j^{\nu_j}=\exp\left\{\sum_{j=1}^m\nu_j\ln p_j\right\},\nonumber\ee
the statistical properties of $n$ w.r.t. $\Pi_m$ are determined by the properties of $\sum_{j=1}^m\nu_j\ln p_j$, which are sums of independent random variables. However the Central Limit Theorem cannot be applied here because $\nu_j$ are not identically distributed. Instead, the following limit theorem is valid.

\begin{theorem}\label{THM1}
Let $\zeta_m=\frac{1}{\ln p_m}\sum_{j=1}^m\nu_j\ln p_j$. As $m\to\infty$ the distributions of $\zeta_m$ converge weakly to the infinitely divisible distribution whose characteristic function $\varphi(\lambda)$ has the form
\be\varphi(\lambda)=\exp\left\{\int_0^1\frac{e^{i\lambda v}-1}{v}\de v\right\}.\label{2}\ee
\end{theorem}
\begin{proof}
The characteristic function $\varphi_m$ of $\zeta_m$ is
\bey \varphi_m(\lambda)&=&\E e^{i\lambda\zeta_m}=\E \exp\left\{\frac{i\lambda}{\ln p_m}\sum_{j=1}^m\nu_j\ln p_j\right\}=\prod_{j=1}^m\left(\frac{p_j}{1+p_j}+\frac{1}{1+p_j}e^{\frac{i\lambda\ln p_j}{\ln p_m}}\right)=\nonumber\\
&=&\prod_{j=1}^m\left(1+\frac{1}{1+p_j}\left(e^{\frac{i\lambda\ln p_j}{\ln p_m}}-1\right)\right)=\nonumber\\
&=&\exp\left\{\sum_{t=1}^{p_m}\left(N(t)-N(t-1)\right)\ln\left(1+\frac{1}{1+t}\left(e^{\frac{i\lambda \ln t}{\ln p_m}}-1\right)\right)\right\}=\nonumber\\
&=&\exp\left\{f_m(p_m+1)N(p_m)-\sum_{t=1}^{p_m}N(t-1)(f_m(t+1)-f_m(t))\right\},\nonumber\eey
by summation by parts, where $f_m(s)=\ln\left(1+\frac{1}{1+s}\left(e^{\frac{i\lambda\ln s}{\ln p_m}}-1\right)\right)$.
Since $f_m$ is complex-valued, the identity $f_m(t+1)-f_m(t)=f_m'(t+\tau)$ for some $0<\tau<1$ does not follow from the mean value theorem and we have to work with the real and imaginary parts separately. Writing $f_m=\Re f_m+i\Im f_m$ we have
\be
\Re f_m(s)=\ln\left|1+\frac{1}{1+s}\left(e^{\frac{i\lambda\ln s}{\ln p_m}}-1\right)\right|=\frac{1}{2}\ln\left(\frac{s^2+2s\cos\left(\frac{\lambda\ln s}{\ln p_m}\right)+1}{(1+s)^2}\right)\nonumber
\ee
and (by choosing the principal branch of the natural logarithm)
\be
\Im f_m(s)=\arg\left(1+\frac{1}{1+s}\left(e^{\frac{i\lambda\ln s}{\ln p_m}}-1\right)\right)
=\arctan\left(\frac{
\sin\left(\frac{\lambda\ln s}{\ln p_m}\right)}{
s+
\cos
\frac{\lambda\ln s}{\ln p_m}
}\right).
\nonumber
\ee
Now, by applying the mean value theorem twice to $\Re f_m$ and $\Im f_m$ separately, we get
\bey 
\Re f_m(t+1)-\Re f_m(t)=(\Re f_m)'(t+\tau_1)=(\Re f_m)'(t)+\tau_1(\Re f_m)''(t+\tau_1')
\nonumber
\eey
for some $0<\tau_1'<\tau_1<1$, and
\bey
\Im f_m(t+1)-\Im(f_m)(t)=(\Im f_m)'(t+\tau_2)=(\Im f_m)'(t)+\tau_2(\Im f_m)''(t+\tau_2')
\nonumber\eey
for some $0<\tau_2'<\tau_2<1$.
Thus
\be
\ln\varphi_m(\lambda)=f_m(p_m+1)N(p_m)-\sum_{t=1}^{p_m}N(t-1)\left(f_m'(t)+\tau_1(\Re f_m)''(t+\tau_1')+\tau_2(\Im f_m)''(t+\tau_2')\right).\label{pf0}
\ee
We claim that the sum involving $f_m'(t)$ gives the main term. In fact,  the first term and the other sums in (\ref{pf0}) tend to zero as $m\to\infty$ (see Appendix). Thus, the main term comes from the following sum:
\bey
&&-\sum_{t=1}^{p_m}N(t-1)f'_m(t)=
-\sum_{t=2}^{p_m}\left(\frac{t}{\ln t}+O\!\left(\frac{t}{\ln^2 t}\right)\right)\frac{1}{1+\frac{1}{1+t}\left(e^{i\lambda\frac{i\lambda\ln t}{\ln p_m}}-1\right)}\cdot\nonumber\\
&&\cdot\left[-\frac{1}{(t+1)^2}\left(e^{i\lambda\frac{i\lambda\ln t}{\ln p_m}}-1\right)+\frac{1}{t(t+1)}e^{i\lambda\frac{i\lambda\ln t}{\ln p_m}}\frac{i\lambda}{\ln p_m}\right]=\nonumber\\
&&=\sum_{t=1}^{p_m}\left(\frac{1}{t\ln t}+O\!\left(\frac{1}{t\ln^2 t}\right)\right)\left(1+\frac{1-e^{\frac{i\lambda\ln t}{\ln p_m}}}{t+e^{\frac{i\lambda\ln t}{\ln p_m}}}\right)\cdot\nonumber\\
&&\cdot\left[\left(e^{\frac{i\lambda\ln t}{\ln p_m}}-1\right)-\frac{(2t+1)\left(e^{\frac{i\lambda\ln t}{\ln p_m}}-1\right)}{(t+1)^2}-\frac{i\lambda}{\ln p_m}\frac{t}{t+1}e^{\frac{i\lambda\ln t}{\ln p_m}}\right]\label{pf1}
\eey
By opening the brackets in (\ref{pf1}) we obtain twelve sums. Let us look at the first sum and consider 
the change of variables (which will be used in the Appendix too) $v=v(t)=\frac{\ln t}{\ln p_m}$  for which  $\de v=v(t+1)-v(t)=v'(t+\tau_3)=v'(t)+\tau_3 v''(t+\tau_3')$ for some $0<\tau_3'<\tau_3<1$. We get

\be\sum_{t=2}^{p_m}\frac{1}{t\ln t}\left(e^{\frac{i\lambda\ln t}{\ln p_m}}-1\right)=\sum_{v}\left(\de v+\frac{\tau}{(t+\tau')^2\ln p_m}\right)\frac{e^{i\lambda v}-1}{v}\longrightarrow\int_0^1\frac{e^{i\lambda v}-1}{v}\de v\nonumber
\ee
as $m\to\infty$ since for some $C>0$
\be\left|\sum_{t=2}^{p_m}\frac{\tau}{(t+\tau')^2\ln p_m}\frac{e^{i\lambda v}-1}{v}\right|\leq \frac{C|\lambda|}{\ln p_m}\sum_{t=2}^{p_m}\frac{1}{t^2}\longrightarrow0.\nonumber
\ee
All the remaining eleven sums coming from (\ref{pf1}) tend to zero (see Appendix) and this concludes the proof of Theorem \ref{THM1}.
\end{proof}
Notice that $$\int\frac{\cos(\lambda v)-1}{v}\de v=-\int_{|\lambda| v}^{\infty}\frac{\cos u}{u}\de u-\ln v\hspace{1cm}\mbox{and}\hspace{1cm}\lim_{x\to0+}\left(-\int_{x}^\infty\frac{\cos u}{u}\de u-\ln x\right)=\gamma,$$
where $\gamma$ is the Euler-Mascheroni constant as before.
Therefore the improper integral $\int_0^1\frac{\cos(\lambda v)-1}{v}\de v$ converges to $-\gamma-\int_{|\lambda|}^\infty\frac{\cos u}{u}\de u-\ln|\lambda|$. On the other hand 
$$\int\frac{\sin(\lambda v)}{v}\de v=\sgn(\lambda)\int_0^{|\lambda|v}\frac{\sin u}{u}\de u\hspace{1cm}\mbox{gives}\hspace{1cm}\int_0^1\frac{\sin(\lambda v)}{v}\de v=\sgn(\lambda)\int_0^{|\lambda|}\frac{\sin u}{u}\de u.$$
This shows that 
\be\varphi(\lambda)=\begin{cases}
      \displaystyle\exp\left\{-\left(\gamma+\int_{\lambda}^\infty\frac{\cos u}{u}\de u+\ln \lambda\right) +i\int_0^\lambda\frac{\sin u}{u}\de u\right\}& \text{$\lambda>0$ }, \\
      &\\
      1& \text{$\lambda$=0 }, \\
      &\\
     \displaystyle\exp\left\{-\left(\gamma+\int_{-\lambda}^\infty\frac{\cos u}{u}\de u+\ln (-\lambda)\right) -i\int_0^{-\lambda}\frac{\sin u}{u}\de u\right\}& \text{$\lambda<0$}.
\end{cases}\nonumber\ee

It is known (see \cite{deBruijn-1951a}) that $\varphi(\lambda)$ is the characteristic function of the \emph{Dickman-De Bruijn distribution}, with density $e^{-\gamma}\rho(t)$, where $\rho(t)$ is 
determined by the initial condition
\be\rho(t)=\begin{cases}0,&t\leq0;\\1,&0<t\leq1,\end{cases}\ee
and 
the integral equation \be t\rho(t)=\int_{t-1}^t\rho(s)\de s,\hspace{.5cm}t\in\R.\nonumber\ee
It also satisfies the delay differential equation \be t \rho'(t)+\rho(t-1)=0\nonumber\ee
for $t\geq1$ (at $t=1$ we consider the right derivative) and for every $k=1,2,3,\ldots$ there is an analytic function $\rho_k(t)$ that gives $\rho(t)$ on $k-1\leq t\leq k$. For example, $\rho_1\equiv1$, $\rho_2(t)=1-\ln t$ and $\rho_3(t)=1-\ln t+\int_2^t\ln(u-1)\frac{\de u}{u}$. It is also easy to see that $\rho\in C^k([k,\infty))$ for each $k$.

Among other properties of $\rho(t)$ one can mention that it is log-concave on $[1,\infty)$ and 
\be
\rho(t)=\exp\left\{-t\left(\ln t+\ln\ln t-1+\frac{\ln\ln t}{\ln t}+O\!\left(\frac{(\ln\ln t)^2}{(\ln t)^2}\right)\right)\right\}
\nonumber\ee
as $t\to\infty$. 
In other words, the limiting density $e^{-\gamma}\rho(t)$ is constant on the interval $(0,1]$, where it takes the value $e^{-\gamma}$, and decays faster then exponentially on $(1,\infty)$, like Poisson distribution. In particular, all its moments exist.

The Dickman-De Bruijn density $\rho$ first appeared in the theory of \emph{smooth numbers} (i.e. numbers with small prime factors). Let $\Psi(x,y)$ denote the number of integers $\leq x$ whose prime factors are $\leq y$. Dickman \cite{Dickman-1930} showed that $\Psi(x,x^{1/u})\sim x\rho(u)$ as $x\to\infty$. The range of $y$ such that the asymptotic formula $\Psi(x,y)\sim x\rho(u)$, where $x=y^u$, has been significantly enlarged by De Bruijn \cite{deBruijn-1951a, deBruijn-1951b, deBruijn-1966} ($y\geq\exp((\ln x)^{5/8+\varepsilon})$) and Hildebrand \cite{Hildebrand-1984a} ($y\geq\exp((\ln\ln x)^{5/3+\varepsilon})$). Notice that in our ensemble $\Omega_m$ (where each element is weighted, not simply counted) we have $x=p_1p_2\cdots p_m$ and $y=p_m$ and thus $y\sim\ln x$. In this regime Erd\"{o}s \cite{Erdos-1963} showed that $\ln\Psi(x,\ln x)\sim \frac{\ln 4\ln x}{\ln\ln x}$ as $x\to\infty$ and therefore the asymptotic is no longer given by the function $\rho$. In other words a \emph{phase transition} occurs in the asymptotic behavior of $\Psi(x,y)$. For a survey on the theoretical and computational aspects of smooth numbers see \cite{Granville2008}.\\

It is worth to mention that in many limit theorems in Number Theory there appear limiting densities which are constants on some interval starting at 0. An example can be found in the work of 
%
Elkies and McMullen \cite{Elkies-McMullen} on the distribution of the gaps in the sequence $\{\sqrt n\mod 1\}$.\\

Here is another example from Probability Theory where the Dickman-De Bruijn distribution appears. Let $\{\eta_j\}_{j\geq1}$ be a sequence of independent random variables such that 
\be P\{\eta_k=k\}=\frac{1}{k}\hspace{1cm}\mbox{and}\hspace{1cm}P\{\eta_k=0\}=1-\frac{1}{k},\nonumber\ee and let $\theta_n=\sum_{j=1}^n\eta_j$ then \be
\lim_{n\to\infty}P\{n^{-1}\theta_n<x\}=e^{-\gamma}\int_0^x\rho(t)\de t.\nonumber\ee

Theorem \ref{THM1} has several important corollaries and applications. An immediate consequence of (\ref{2}) is that
\be\Pi_m\{n\leq p_m^s\}=\sum_{n\leq p_m^s,\,n\in\Omega_m}\pi_m(x)\longrightarrow e^{-\gamma}\int_0^s\rho(t)\de t\nonumber\ee
as $m\to\infty$. For instance, for $s=2$ we get $e^{-\gamma}(3-\ln4)\approx0.90603$. In other words, despite the fact that the largest element of our ensemble $\Omega_m$ is of order $m^m$, approximately 90\% of the ``mass'' of our probability distribution $\Pi_m$ is concentrated on numbers less than $p_m^2$ for large $m$.\\


Let us fix $0<\sigma\leq1$ and decompose the interval $(0,\sigma)$ onto $K$ equal intervals $(\delta_k,\delta_{k+1})$, $\delta_k=\frac{\sigma k}{K}$, $k=0,
\ldots,K-1$. For fixed $K$, Theorem \ref{THM1} states that
\be\Pi_m\left\{\delta_k<\frac{\ln n}{\ln p_m}<\delta_{k+1}\right\}\longrightarrow\frac{e^{-\gamma}\sigma}{K}\label{thm1-for-interval-K}\ee
as $m\to\infty$. Let us consider the error term in (\ref{thm1-for-interval-K})
\be E_m^{(\sigma)}(k,K):=\Pi_m\left\{\delta_k<\frac{\ln n}{\ln p_m}<\delta_{k+1}\right\}-\frac{e^{-\gamma}\sigma}{K}.\nonumber\ee
In the rest of this paper we provide some estimates about the error terms $E_m^{(\sigma)}(k,K)$ when $K$ grows with $n$.
We prove the following
\begin{theorem}\label{THM2}
For every $\varepsilon>0$ and every function $K(m)$ such that $\lim_{m\to\infty}\frac{\ln^3p_m}{K(m)^2}=c\geq0$ there exists $m^*=m^*(\varepsilon,K)$ such that the inequalities
\be
-\frac{c\sigma^3}{12\zeta(2)}-\varepsilon\leq\frac{Z_m}{p_m^\sigma}\sum_{k=0}^{K(m)-1}p_m^{\delta_k}E_m^{(\sigma)}(k,K(m))\leq
\frac{c\sigma^3}{12\zeta(2)}+\varepsilon\label{statement-THM2}\ee
hold for every $m\geq m^*$ and every $0<\sigma\leq 1$.
\end{theorem}

An important tool in the proof of Theorem \ref{THM2} is given by the counting function $$M_m(t)=\#\left\{n\leq t:\: n\in\Omega_m\right\}.$$ This is analogous to the classical quantity 
$$M(t)=\#\left\{n\leq t:\:\mu(n)\neq0\right\},$$ for which the asymptotic
\be\lim_{t\to\infty}\frac{M(t)}{t}=\frac{1}{\zeta(2)}=\frac{6}{\pi^2}\approx0.607927.\label{1/zeta2}\nonumber\ee
holds (see, e.g., \cite{Montgomery-Vaughan}).
%
Even though the ensemble $\Omega_m$ is very sparse, 
 its initial segment of length $p_m$ 
 contains all square-free numbers less or equal than $p_m$. In particular $\lim_{m\to\infty}\frac{M_m(p_m^\sigma)}{p_m^\sigma}=\frac{1}{\zeta(2)}$ for every $0<\sigma\leq1$.  For $\sigma=1$ this fact can be rephrased as 
$$\lim_{m\to\infty}\frac{1}{p_m}\sum_{n\leq p_m}\mu^2(n)=\frac{1}{\zeta(2)}$$ 
and can be compared with 
$$\lim_{m\to\infty}\frac{1}{\ln p_m}\sum_{n\leq p_m}\frac{\mu^2(n)}{n}=e^{-\gamma},$$
which is a corollary of our Theorem \ref{THM1}.\\ 
%

The following Lemma provides some simple estimates that will be used in the proof of  Theorem \ref{THM2}.
\begin{lem}\label{lemma}
\be0<\sum_{k=0}^{K-1}p_m^{\delta_{k+1}}\frac{\sigma}{K}-\frac{p_m^\sigma-1}{\ln p_m}\leq\frac{\sigma^3 p_m^\sigma \ln^2 p_m}{12 K^2}+\frac{\sigma(p_m^\sigma-1)}{2K}\label{lem-inequalities}\ee
\be  -\frac{\sigma^3 p_m^\sigma \ln^2 p_m}{12 K^2}-\frac{\sigma(p_m^\sigma-1)}{2K} \leq\sum_{k=0}^{K-1}p_m^{\delta_{k}}\frac{\sigma}{K}-\frac{p_m^\sigma-1}{\ln p_m}<0\label{lem-inequalities2}\ee
\end{lem}
\begin{proof}
The right (resp. left) Riemann sum $\sum_{k=0}^{K-1}p_m^{\delta_{k+1}}\frac{\sigma}{K}$ (resp. $\sum_{k=0}^{K-1}p_m^{\delta_{k}}\frac{\sigma}{K}$) converges as $K\to\infty$ to the integral $\int_0^\sigma e^{\delta\ln p_m}\de\delta=\frac{p_m^\sigma-1}{\ln p_m}$. Moreover, since the function $t\mapsto p_m^t$ is increasing, the right (resp. left) sum is strictly bigger (resp. smaller) than the integral. This proves the first inequality in (\ref{lem-inequalities}) and the second inequality in (\ref{lem-inequalities2}). 
A classical result from Calculus states that in the absolute value of the error performed by approximating the integral $\int_a^b f(x)\de x$ by the trapezoidal Riemann sum $$\left(\ha f(x_0)+f(x_1)+f(x_2)+\ldots+f(x_{K-1})+\ha f(x_K)\right)\Delta x,$$ $x_k=a+k\frac{b-a}{K}$ is bounded by $\frac{M(b-a)^3}{12K^2}$ where $\sup_{a\leq x\leq b}|f''(x)|\leq M$.
This implies that the error for the right Riemann sum $$\left(f(x_1)+\ldots+f(x_{K})\right)\Delta x$$ is bounded from above by $\frac{M(b-a)^3}{12K^2}+(f(b)-f(a))\frac{b-a}{2K}$ and gives the second inequality of (\ref{lem-inequalities}) when applied to the function $t\mapsto p_m^t$ over the interval $[0,\sigma]$.
On the other hand, the error given by the left Riemann sum  $$\left(f(x_0)+\ldots+f(x_{K-1})\right)\Delta x$$ 
is bounded from below by $-\frac{M(b-a)^3}{12K^2}-(f(b)-f(a))\frac{b-a}{2K}$ and this gives the first inequality in (\ref{lem-inequalities2}).
\end{proof}

\begin{proof}[Proof of Theorem \ref{THM2}]
\bey
\frac{M_m(p_m^\sigma)}{p_m^\sigma}=\frac{Z_m}{p_m^\sigma}\sum_{\scriptsize{\ba{c}n\in\Omega_m\\n\leq p_m^\sigma\ea}}n\pi_m(n)=\frac{Z_m}{p_m^\sigma}\sum_{k=0}^{K-1}\sum_{\scriptsize{\ba{c}n\in\Omega_m\\p_m^{\delta_k}< n\leq p_m^{\delta_{k+1}}\ea}}n\pi_m(n)\leq\nonumber\\
\leq\frac{Z_m}{p_m^\sigma}\sum_{k=0}^{K-1}p_m^{\delta_{k+1}}\sum_{\scriptsize{\ba{c}n\in\Omega_m\\p_m^{\delta_k}< n\leq p_m^{\delta_{k+1}}\ea}}\pi_m(n)=\frac{Z_m}{p_m^\sigma}e^{-\gamma}\sum_{k=0}^{K-1}p_m^{\delta_{k+1}}\frac{\sigma}{K}+\frac{Z_m}{p_m^\sigma}\sum_{k=0}^{K-1}p_m^{\delta_{k+1}}E^{(\sigma)}_m(k,K)\nonumber
\eey
Applying Lemma \ref{lemma} to the right Riemann sum $\sum_{k=0}^{K-1}p_m^{\delta_{k+1}}\frac{\sigma}{K}$ we obtain the estimate
\bey
\frac{M_m(p_m^\sigma)}{p_m^\sigma}\leq \frac{e^{-\gamma}Z_m}{\ln p_m}\frac{p_m^\sigma-1}{p_m^\sigma}+\frac{e^{-\gamma}Z_m}{\ln p_m}\left(\frac{\sigma^3\ln^3p_m}{12 K^2}+\frac{p_m^\sigma-1}{p_m^\sigma}\frac{\sigma\ln p_m}{2K}\right)+\frac{Z_m}{p_m^\sigma}\sum_{k=0}^{K-1}p_m^{\delta_{k+1}}E^{(\sigma)}_m(k,K)\nonumber
\eey
which is true for every $m$ and $K$.
Since, as $m\to\infty$, $\frac{M_m(p_m^\sigma)}{p_m^\sigma}\to\frac{1}{\zeta(2)}$, $\frac{Z_m}{\ln p_m}\to\frac{e^\gamma}{\zeta(2)}$, and by hypothesis $\frac{\ln^3 p_m}{K(m)^2}\to c$ (and thus $\frac{\ln p_m}{K(m)}\to0$), then for every $\varepsilon>0$ the inequality
\be\frac{Z_m}{p_m^\sigma}\sum_{k=1}^{K(m)-1}p_m^{\delta_{k+1}}E^{(\sigma)}_m(k,K(m))\geq
-\frac{c\sigma^3}{12\zeta(2)}-\varepsilon\nonumber
\ee
holds true for sufficiently large $m$. By noticing that $p_m^{\delta_{k+1}}=p_m^{\delta_k}\left(1+(e^{\frac{\sigma\ln p_m}{K(m)}}-1)\right)$ and $0\leq (e^{\frac{\sigma\ln p_m}{K(m)}}-1)\to0$ as $m\to\infty$, we obtain the first inequality of (\ref{statement-THM2}). 
On the other hand
\be\frac{M_m(p_m^\sigma)}{p_m^\sigma}\geq\frac{Z_m}{p_m^\sigma}\sum_{k=0}^{K-1}p_m^{\delta_k}\sum_{\scriptsize{\ba{c}n\in\Omega_m\\p_m^{\delta_k}< n\leq p_m^{\delta_{k+1}}\ea}}\pi_m(n)=\frac{Z_m}{p_m^\sigma}e^{-\gamma}\sum_{k=0}^{K-1}p_m^{\delta_k}\frac{\sigma}{K}+\frac{Z_m}{p_m^\sigma}\sum_{k=0}^{K-1}p_m^{\delta_k}E_m^{(\sigma)}(k,K)\nonumber\ee
and applying Lemma \ref{lemma} to the left Riemann sum $\sum_{k=0}^{K-1}p_m^{\delta_k}\frac{\sigma}{K}$ we obtain the estimate
\be\frac{M_m(p_m^\sigma)}{p_m^\sigma}\geq \frac{e^{-\gamma}Z_m}{\ln p_m}\frac{p_m^\sigma-1}{p_m^\sigma}-\frac{e^{-\gamma}Z_m}{\ln p_m}\left(\frac{\sigma^3\ln^3p_m}{12K^2}+\frac{p_m^\sigma-1}{p_m^\sigma}\frac{\sigma\ln p_m}{2K}\right)+\frac{Z_m}{p_m^\sigma}\sum_{k=0}^{K-1}p_m^{\delta_k}E_m^{(\sigma)}(k,K)\nonumber
\ee
which is true for every $m$ and $K$. Proceeding as above we have that for every $\varepsilon>0$ the inequality
\be\frac{Z_m}{p_m^\sigma}\sum_{k=0}^{K(m)-1}p_m^{\delta_k}E_m^{(\sigma)}(k,K(m))\leq\frac{c\sigma^3}{12\zeta(2)}+\varepsilon
\nonumber\ee
holds for sufficiently large $m$ and we have 
the second inequality of (\ref{statement-THM2}).
\end{proof}
An immediate consequence of 
Theorem \ref{THM2} is the following
\begin{cor}\label{cor1}
Consider a function $K(m)$ such that $\lim_{m\to\infty}\frac{\ln^3p_m}{K(m)^2}=c\geq0$. Then the sum 
of the error terms coming from (\ref{thm1-for-interval-K}), with weights $p_m^{-\sigma+\delta_k}$, satisfies the asymptotic estimate
\be\sum_{k=0}^{K(m)-1}\frac{E_m^{(\sigma)}(k,K(m))}{p_m^{\sigma-\delta_k}}=\begin{cases}O\!\left(\frac{1}{\ln p_m}\right)&\mbox{if $c>0$;}\\ &\\o\!\left(\frac{1}{\ln p_m}\right)&\mbox{if $c=0$;}\end{cases}\label{statement-cor1}
\ee
 as $m\to\infty$ for every $0<\sigma\leq1$.
\end{cor}
Notice that implied constant in the $O$-notation depends explicitly on $c$ and $\sigma$ by (\ref{asymptotic-Zm}) and (\ref{statement-THM2}). Moreover, as $k$ ranges from $0$ to $K(m)-1$, the weights vary from $p_m^{-\sigma}$ ($\to0$ as $m\to\infty$) to $e^{-\frac{\ln p_m}{K(m)}}$ ($\to1$ as $m\to\infty$). This means that the error terms $E_m^{(\sigma)}(k,K(m))$ corresponding to small values of $k$ are allowed to be larger in absolute value.

 In order to get estimates on the mean value of the error term (for which al weights are equal to $\frac{1}{K(m)}$) we just replace the weights $p_m^{-\sigma+\delta_k}$ by either $p^{-\sigma}$ or $1$ in (\ref{statement-THM2}). This yields, for every $\varepsilon$ and sufficiently large $m$,
\be\frac{1}{Z_m}\left(-\frac{c\sigma^3}{12\zeta(2)}-\varepsilon\right)\leq\sum_{k=0}^{K(m)-1}E_m^{(\sigma)}(k,K(m))\leq\frac{p_m^\sigma}{Z_m}\left(\frac{c\sigma^3}{12\zeta(2)}+\varepsilon\right).\nonumber
\ee
In particular 
we get, as $m\to\infty$, 
\be\langle E_m^{(\sigma)}\rangle:=\frac{1}{K(m)}\sum_{k=0}^{K(m)-1}E_m^{(\sigma)}(k,K(m))=\begin{cases}O\!\left(\frac{p_m^\sigma}{\ln^{5/2}p_m}\right)&\mbox{if $c>0$;}\\&\\o\!\left(\frac{p_m^\sigma}{K(m)\ln p_m}\right)&\mbox{if $c=0$.}\end{cases}\nonumber\ee
Let us point out that, even though by (\ref{thm1-for-interval-K}) the error term $E_m^{(\sigma)}(k,K(m))$ tends to zero as $m\to\infty$ for each $k$, it is not a priori true that $\langle E_m^{(\sigma)}\rangle$ tends to zero as well. It follows from our Theorem \ref{THM2} that this is indeed the case when
$\frac{p_m^\sigma}{K(m)\ln p_m}$ remains bounded (i.e. a particular case of $c=0$). 
Let us summarize this fact in the following
\begin{cor}
Let $0<\sigma\leq1$ and consider a function $K(m)$ such that $\lim_{m\to\infty}\frac{p_m^\sigma}{K(m)\ln p_m}<\infty$. Then, as $m\to\infty$,
\be\langle E_m^{(\sigma)}\rangle=o\!\left(\frac{p_m^\sigma}{K(m)\ln p_m}\right).\label{statement-cor2}\ee
\end{cor}
In other words, if $K$ grows sufficiently fast (namely as $\mbox{\emph{const}}\cdot\frac{p_m^\sigma}{\ln p_m}$ or faster), then the mean value of the error $\langle E_m^{(\sigma)}\rangle$ tends to zero as $m\to\infty$ and the rate of convergence to zero is controlled explicitly in terms of $\sigma$ and $K$. 

Notice that one would expect the error term $E_m^{(\sigma)}(k,K(m))$ in (\ref{thm1-for-interval-K}) to be $o\!\left(\frac{1}{K(m)}\right)$, 
however we could only derive the weaker asymptotic estimates (\ref{statement-cor1}) and (\ref{statement-cor2}) from Theorem \ref{THM2}. A possible approach to further investigate the size of the error term in (\ref{thm1-for-interval-K}) would be to first prove an analogue of Theorem \ref{THM1} for shrinking intervals. This is, however, beyond the aim of this paper. 

\section*{Appendix}
This Appendix contains the estimates for the error terms in the proof Theorem \ref{THM1}. By $C_j$, $j=1,\ldots,21$, we will denote some positive constants.

The first term of (\ref{pf0}) tends to zero as $m\to\infty$ uniformly in $\lambda$. In fact using 
(\ref{strongPNT})
we obtain
\bey
&&\Re f_m(p_m+1)N(p_m)=\frac{N(p_m)}{2}\ln\left(\frac{(p_m+1)^2+2(p_m+1)\cos\left(\frac{\lambda\ln(p_m+1)}{\ln p_m}\right)+1}{(p_m+2)^2}\right)=\nonumber\\
&&=\frac{N(p_m)}{2}\left(\ln\left(1+O\!\left(\frac{1}{p_m}\right)\right)-\ln\left(1+O\!\left(\frac{1}{p_m}\right)\right)\right)
=O\!\left(\frac{p_m}{\ln p_m}\right)O\!\left(\frac{1}{p_m}\right)=O\!\left(\frac{1}{\ln p_m}\right),\nonumber\eey
and 
\bey
\Im f_m(p_m+1)N(p_m)&=&N(p_m)\arctan\left(\frac{\sin\left(\lambda\frac{\ln(p_m+1)}{\ln p_m}\right)}{p_m+1+\cos\left(\lambda\frac{\ln(p_m+1)}{\ln p_m}\right)}\right)=\nonumber\\
&=&O\!\left(\frac{p_m}{\ln p_m}\right)O\!\left(\frac{1}{p_m}\right)=O\!\left(\frac{1}{\ln p_m}\right)\nonumber
\eey
as $m\to\infty$, and the implied constants do not depend on $\lambda$.
An explicit computation shows that
\be
(\Re f_m)''(s)=f^{(1)}_m(s)+f^{(2)}_m(s)+f^{(3)}_m(s),\nonumber
\ee
where
\bey
f^{(1)}_m(s)&=&-\lambda^2\frac{2s+(1+s^2)\cos\left(\frac{\lambda\ln s}{\ln p_m}\right)}{s\left(s^2+2s\cos\left(\frac{\lambda\log s}{\log p_m}\right)+1\right)^2\ln^2p_m},\nonumber\\
f^{(2)}_m(s)&=&\lambda\frac{\left(3s^2+2s\cos\left(\frac{\lambda\ln s}{\ln p_m}\right)-1\right)\sin\left(\frac{\lambda\ln s}{\ln p_m}\right)}{s\left(s^2+2s\cos\left(\frac{\lambda\log s}{\log p_m}\right)+1\right)^2\ln p_m},\nonumber\\
f^{(3)}_m(s)&=&\frac{2\left(\cos\left(\frac{\lambda\ln s}{\ln p_m}\right)-1\right)\left(s^3-s^2-s-1+(s^2-2s-1)\cos\left(\frac{\lambda\ln s}{\ln p_m}\right)\right)}{(1+s)^2\left(s^2+2s\cos\left(\frac{\lambda\log s}{\log p_m}\right)+1\right)^2}\nonumber.
\eey
We have
\be
\left|f^{(1)}_m(s)\right|\leq \frac{C_1\lambda^2}{s^3\ln^2p_m},\hspace{1cm}\left|f^{(2)}_m(s)\right|\leq\frac{C_2|\lambda|}{s^3\ln p_m}\nonumber
\ee
and thus
\bey
&&\left|\sum_{t=1}^{p_m}N(t-1)\tau_1f^{(1)}_m(t+\tau_1')\right|\leq\frac{C_3\lambda^2}{\ln^2 p_m}\sum_{t=2}^{p_m}\frac{1}{t^2\ln t}\longrightarrow0\hspace{.5cm}\mbox{and 
}\nonumber\\
&&\left|\sum_{t=1}^{p_m}N(t-1)\tau_1f^{(2)}_m(t+\tau_1')\right|\leq\frac{C_4|\lambda|}{\ln p_m}\sum_{t=2}^{p_m}\frac{1}{t^2\ln t}\longrightarrow0\hspace{.5cm}\mbox{as $m\to\infty$}.\nonumber
\eey
The third function satisfies the estimate
\be
\left|f^{(3)}_m(s)\right|\leq\frac{s^3\left|2\cos\left(\frac{\lambda\ln s}{\ln p_m}\right)-2\right|+s^2C_5}{(1+s)^2(1-s)^4}\leq\frac{C_6\left(1-\cos\left(\frac{\lambda\ln s}{\ln p_m}\right)\right)}{s^3}.\nonumber
\ee
We now perform the same change of variables 
$v=v(t)=\frac{\ln t}{\ln p_m}$ as before (using $\tau_3$ and $\tau_3'$ as in the proof of Theorem \ref{THM1}). 
We get
\bey
&&\left|\sum_{t=1}^{p_m}N(t-1)\tau_1f^{(3)}_m(t+\tau_1')\right|\leq C_7\sum_{t=2}^{p_m}\frac{1-\cos\left(\frac{\lambda\ln t}{\ln p_m}\right)}{t^2\ln t}\leq\nonumber\\
&&\leq C_8\sum_{v}\left(\de v+\frac{\tau_3}{(t+\tau_3')^2\ln p_m}\right)\frac{1-\cos(\lambda u)}{t\,v}\longrightarrow0\nonumber
\eey
as $m\to\infty$.
Another explicit computation shows that 
\be
(\Im f_m)''(s)=f^{(4)}_m(s)+f^{(5)}_m(s)+f^{(6)}_m(s),\nonumber
\ee
where
\bey
f^{(4)}_m(s)&=&\lambda^2\frac{(s^2-1)\sin\left(\frac{\lambda\log s}{\log p_m}\right)}{s\left(s^2+2s\cos\left(\frac{\lambda\log s}{\log p_m}\right)+1\right)^2\ln^2p_m},\nonumber\\
f^{(5)}_m(s)&=&-\lambda\frac{1+5s^2+2s^2\cos^2\left(\frac{\lambda\log s}{\log p_m}\right)+(3s^3+5s)\cos\left(\frac{\lambda\log s}{\log p_m}\right)}{s^2\left(s^2+2s\cos\left(\frac{\lambda\log s}{\log p_m}\right)+1\right)^2\ln p_m}\nonumber\\
f^{(6)}_m(s)&=&\frac{2\left(s+\cos\left(\frac{\lambda\log s}{\log p_m}\right)\right)\sin\left(\frac{\lambda\log s}{\log p_m}\right)}{\left(s^2+2s\cos\left(\frac{\lambda\log s}{\log p_m}\right)+1\right)^2}.\nonumber
\eey
We have the estimates
\be
\left|f^{(4)}_m(s)\right|\leq\frac{C_{10}\lambda^2}{s^3\ln^2 p_m},\hspace{1cm}\left|f^{(5)}_m(s)\right|\leq\frac{C_{11}|\lambda|}{s^3\ln p_m}\nonumber
\ee
and thus
\bey
&&\left|\sum_{t=1}^{p_m}N(t-1)\tau_1f^{(4)}_m(t+\tau_1')\right|\leq\frac{C_{12}\lambda^2}{\ln^2 p_m}\sum_{t=2}^{p_m}\frac{1}{t^2\ln t}\longrightarrow0\hspace{.5cm}\mbox{and 
}\nonumber\\
&&\left|\sum_{t=1}^{p_m}N(t-1)\tau_1f^{(5)}_m(t+\tau_1')\right|\leq\frac{C_{13}|\lambda|}{\ln p_m}\sum_{t=2}^{p_m}\frac{1}{t^2\ln t}\longrightarrow0\hspace{.5cm}\mbox{as $m\to\infty$}.\nonumber
\eey
The estimate
\be\left|f^{(6)}_m(s)\right|\leq\frac{ C_{14}s \sin\left(\frac{\lambda\log s}{\log p_m}\right)}{(s-1)^4}\leq \frac{C_{15}\sin\left(\frac{\lambda\log s}{\log p_m}\right)}{s^3}\nonumber
\ee
yields, as $m\to\infty$,
\be
\left|\sum_{t=1}^{p_m}N(t-1)\tau_1f^{(6)}_m(t+\tau_1')\right|\leq C_{15}\sum_{t=2}^{p_m}\frac{\sin\left(\frac{\lambda\ln t}{\ln p_m}\right)}{t^2\ln t}\leq C_{16}\sum_{v}\left(\de v+\frac{\tau_3}{(t+\tau_3')^2\ln p_m}\right)\frac{\sin(\lambda v)}{t\,v}
\rightarrow0.\nonumber
\ee
This concludes the analysis of the error terms coming from (\ref{pf0}). 

Let us now deal with the error terms coming from (\ref{pf1}). One sum (giving the main term) is already discussed in the proof of Theorem \ref{THM1}. Amongst the remaining eleven sums coming from (\ref{pf1}), it is enough to check that the following three tend to zero as $m\to\infty$ (the other eight being dominated by these):
\bey
&&\left|\sum_{t=2}^{p_m}\frac{1}{t\ln t}\frac{(2t-1)\left(e^{\frac{i\lambda\ln t}{\ln p_m}}-1\right)}{(t+1)^2}\right|\leq C_{17}\sum_{t=1}^{p_m}\left|\frac{e^{\frac{i\lambda\ln t}{\ln p_m}}-1}{\frac{\ln t}{\ln p_m}}\right|\frac{1}{t^2\ln p_m}\leq\frac{C_{18}|\lambda|}{\ln p_m}\sum_{t=2}^{p_m}\frac{1}{t^2}\longrightarrow0,\nonumber\\
&&\left|\frac{i\lambda}{\ln m}\sum_{t=2}^{p_m}\frac{1}{t\ln t}\frac{t}{t+1}e^{\frac{i\lambda\ln t}{\ln p_m}}\right|\leq \frac{C_{19}|\lambda|}{\ln m}\sum_v\left(\de v+\frac{\tau}{(t+\tau')^2\ln p_m}\right)\frac{e^{i\lambda v}}{v}\longrightarrow0,\nonumber\\
&&\left|\sum_{t+2}^{p_m}\frac{1}{t\ln t}\frac{\left(e^{\frac{i\lambda\ln t}{\ln p_m}}-1\right)^2}{t+e^{\frac{i\lambda \ln t}{\ln p_m}}}\right|\leq C_{20}\sum_{t=2}^{p_m}\left|\frac{\left(e^{\frac{i\lambda\ln t}{\ln p_m}}-1\right)^2}{\left(\frac{\ln t}{\ln p_m}\right)^2}\right|\frac{\ln t}{t^2\ln^2 p_m}\leq\frac{C_{21}\lambda^2}{\ln^2p_m}\sum_{t=2}^{p_m}\frac{\ln t}{t^2}\longrightarrow0.\nonumber
\eey

\section*{Acknowledgments}
We would like to thank Alex Kontorovich and Andrew Granville for useful discussions and comments. The second author acknowledges the financial support from the NSF Grant 0600996. 

\addcontentsline{toc}{section}{Bibliography}
\bibliographystyle{plain}
\bibliography{bibliography-mobius}
\end{document}